\documentclass{article}

\usepackage{amsmath,amsfonts,amssymb,url,amsthm,mathtools,tikz}

\usepackage[utf8]{inputenc}

\newtheorem{proposition}{Proposition}[section]
\newtheorem{lemma}[proposition]{Lemma}
\newtheorem{corollary}[proposition]{Corollary}

\newtheorem{theorem}[proposition]{Theorem}

\newtheorem{remark}[proposition]{Remark}

\newtheorem{example}[proposition]{Example}

\newcommand{\eps}{\varepsilon}

\newcommand\C{{\mathbb C}}

\newcommand\Q{{\mathbb Q}}
\newcommand\R{{\mathbb R}}
\newcommand\Z{{\mathbb Z}}

\newcommand{\calD}{{\mathcal{D}}}
\newcommand{\calF}{{\mathcal{F}}}

\newcommand{\calH}{\mathcal{H}}

\newcommand{\OO}{{\mathcal{O}}}

\renewcommand{\Im}{{\operatorname{Im}}}
\newcommand{\Gal}{{\operatorname{Gal}}}

\renewcommand{\Re}{{\operatorname{Re}}}

\newcommand{\ord}\nu

\title{Separating singular moduli and the primitive element problem}

\author{Yuri BILU, Bernadette FAYE, Huilin ZHU\footnote{corresponding author}}

1

\makeatletter

\renewcommand*\l@section[2]{%
  \ifnum \c@tocdepth >\z@
    \addpenalty\@secpenalty
    \addvspace{0.2em \@plus\p@}%
    \setlength\@tempdima{1.5em}%
    \begingroup
      \parindent \z@ \rightskip \@pnumwidth
      \parfillskip -\@pnumwidth
      \leavevmode \bfseries
      \advance\leftskip\@tempdima
      \hskip -\leftskip
      #1\nobreak\hfil \nobreak\hb@xt@\@pnumwidth{\hss #2}\par
    \endgroup
  \fi}

\makeatother

\numberwithin{equation}{section}

\begin{document}

\hbadness 300

\hfuzz 5pt

\maketitle

\begin{abstract}
We prove that ${|x-y|\ge 800X^{-4}}$, where~$x$ and~$y$ are distinct singular moduli of discriminants not exceeding~$X$.  We apply this result to the ``primitive element problem'' for two singular moduli.
In a previous article Faye and Riffaut show that the number field  $\Q(x,y)$, generated by two distinct singular moduli~$x$ and~$y$, is generated by ${x-y}$ and, with some exceptions, by ${x+y}$ as well. In this article we fix a rational number
${\alpha \ne0,\pm1}$  and show that the field  $\Q(x,y)$ is generated by ${x+\alpha y}$, with a few exceptions occurring when~$x$ and~$y$ generate the same quadratic field over~$\Q$.   Together with the above-mentioned result of Faye and Riffaut, this generalizes a theorem due to Allombert et al. (2015) about solutions of linear equations in singular moduli.
\end{abstract}

{\footnotesize

\tableofcontents

}

\section{Introduction}
A \textsl{singular modulus} is the $j$-invariant of an elliptic curve with complex multiplication. Given a singular modulus~$x$, we denote by $\Delta_x$ the discriminant of the associated imaginary quadratic order and  by $h(\Delta)$ the class number of the imaginary quadratic order of discriminant~$\Delta$. Recall that two singular moduli~$x$ and~$y$ are conjugate over~$\Q$ if and only if ${\Delta_x=\Delta_y}$, and that all singular moduli of  given discriminant~$\Delta$ form a full Galois orbit over~$\Q$. In particular, ${[\Q(x):\Q]=h(\Delta_x)}$. For all details, see, for instance, \cite[\S 7 and \S 11]{Co13}.

Lower estimates for a non-zero
singular modulus play a crucial role in some recent works on Diophantine properties
of singular moduli. For example, in \cite{BMZ13,Ha15,BLP16,BHK18} the authors obtain and use some lower
bounds of the shape ${|x|\gg |\Delta_x|^{-c}}$ with some absolute explicit constant ${c>0}$. 

In this article we obtain a totally explicit lower bound for the difference ${|x-y|}$, where~$x$ and~$y$ are distinct singular moduli. Since~$0$ is a singular modulus, this generalizes the previous lower bounds for $|x|$.

\begin{theorem}
\label{thseparweak}
Let~$x$ and~$y$ be two distinct singular moduli. Then,
$$
|x-y|\ge 800\max\{|\Delta_x|,|\Delta_y|\}^{-4}.
$$
\end{theorem}

In fact, we obtain a more precise statement, see Theorem~\ref{thsepar}.

We apply Theorem~\ref{thseparweak} to the ``primitive element problem'' for singular moduli. It is well known that, given a field~$k$  of characteristic~$0$ and~$x,y$ algebraic over~$k$,  the field $k(x,y)$ has a generator (called sometimes ``primitive element'') of the form ${x+\alpha y}$, where ${\alpha\in k}$.  Moreover, any non-zero~$\alpha$ would work with finitely many exceptions, and often this set of exceptions is empty.

We consider the case ${k=\Q}$ and $x,y$ singular moduli, and we study the question ``does ${x+\alpha y}$ generate $\Q(x,y)$ for every ${\alpha\in \Q^\times}$?''. To motivate this question, we recall that, starting from the ground-breaking article of André~\cite{An98}, equations involving singular moduli were studied by many authors, see~\cite{ABP15,BLP16,Ri19} for a historical account and further references. In particular, Kühne~\cite{Ku13} proved that the equation ${x+y=1}$ has no solutions in singular moduli~$x$ and~$y$.
This was generalized in~\cite{ABP15}, where  solutions in singular moduli of a general linear equation with rational coefficients are classified.

\begin{theorem} \textrm{\cite[Theorem~1.2]{ABP15}}
\label{thspinq}
Let $A,B,C$ be rational numbers such that ${AB\ne 0}$.  
Let~$x$ and~$y$ be singular moduli such that ${Ax+By=C}$. Then either ${A+B=C=0}$ and ${x=y}$, or the field ${\Q(x)=\Q(y)}$ is of degree at most~$2$ over~$\Q$.
\end{theorem}

Note that lists of all imaginary quadratic discriminants~$\Delta$ with ${h(\Delta)\le 2}$ are widely available, so Theorem~\ref{thspinq} is fully explicit.

One can re-state Theorem~\ref{thspinq} as follows.

\paragraph{Theorem~\ref{thspinq}${}'$.} \textit{Let~$\alpha$ be a non-zero rational number, and let $x,y$ be singular moduli such that ${x+\alpha y\in \Q}$. Then either ${\alpha=-1}$ and ${x=y}$ or the field ${\Q(x)=\Q(y)}$ is of degree at most~$2$ over~$\Q$.}

\bigskip

This raises the following natural question: what is the number field generated by ${x+\alpha y}$? This is clearly a subfield of $\Q(x,y)$, and one may wonder how smaller than $\Q(x,y)$ this field is. The problem is trivial when ${x=y}$, so we may assume that ${x\ne y}$.

In the special case when ${\alpha =\pm1}$ this question was addressed in~\cite{FR18}. It turns out that ${x-y}$ always generates $\Q(x,y)$, and ${x+y}$ generates a subfield of $\Q(x,y)$ of degree at most~$2$, which is often $\Q(x,y)$ itself.   To be precise, we have the following statement.

\begin{theorem} \textrm{\cite[Theorem 4.1]{FR18}}
\label{thsumdiff}
Let $x,y$ be two distinct singular moduli and let ${\alpha\in\{\pm1\}}$. Then ${\Q(x,y)=\Q(x+\alpha y)}$, unless ${\alpha=1}$ and ${\Delta_x=\Delta_y}$, in which case we have ${[\Q(x,y):\Q(x+y)]\le 2}$.
\end{theorem}

In the present article we treat the case ${\alpha \ne \pm1}$. There is one obvious case, given by the following example, when ${x+\alpha y}$ does not generate $\Q(x,y)$.

\begin{example}
\label{exquad}
Let~$x$ and~$y$ generate the same number field of degree~$2$ over~$\Q$, and denote $x',y'$ their respective conjugates over~$\Q$. Set
\begin{equation}
\label{ealpha}
\alpha=-\frac{x-x'}{y-y'}.
\end{equation}
Then  ${\alpha \in \Q}$ and ${x+\alpha y\in \Q}$; hence ${x+\alpha y}$ cannot generate the quadratic field $\Q(x,y)$.

Note that, when  ${\Delta_x=\Delta_y}$, then ${\alpha=1}$, and we are in a special case of Theorem~\ref{thsumdiff}. On the other hand, if ${\Delta_x\ne \Delta_y}$ then ${\alpha \ne \pm1}$ by Theorem~\ref{thsumdiff}.
\end{example}

All cases of Example~\ref{exquad} can be easily listed using the available lists of imaginary quadratic discriminants of class number~$2$.

Our principal result tells that Example~\ref{exquad} lists all cases when ${x+\alpha y}$ is not a primitive element of $\Q(x,y)$.

\begin{theorem}
\label{thalpha}
Let ${\alpha \ne 0, \pm1}$ be a rational number and $x,y$ singular moduli. Then
either ${\Q(x+\alpha y)=\Q(x,y)}$,
or $x,y,\alpha$ are as in Example~\ref{exquad}, that is
\begin{enumerate}

\item
$x$ and~$y$ generate the same number field, which is of degree~$2$ over~$\Q$;
\item
$\Delta_x\ne \Delta_y$;

\item
$\alpha=-(x-x')/(y-y')$, where $x',y'$ are the conjugates of $x,y$  over~$\Q$.

\end{enumerate}
\end{theorem}
Note that we do not assume ${x\ne y}$, because the statement holds trivially for ${x=y}$.

Together with Theorem~\ref{thsumdiff} this generalizes Theorem~\ref{thspinq}${}'$.



\subsection{General conventions}
\label{ssprelim}
Unless the contrary is stated explicitly,  the letter~$\Delta$ stands for an \textsl{imaginary quadratic discriminant}, that is, ${\Delta<0}$ and  ${\Delta\equiv 0,1\bmod 4}$. 

We denote by $\OO_\Delta$ the imaginary quadratic order of discriminant~$\Delta$, that is, ${\OO_\Delta =\Z[(\Delta+\sqrt\Delta)/2]}$. Then ${\Delta=Df^2}$, where~$D$ is  discriminant of the number field ${K=\Q(\sqrt\Delta)}$ (called the \textsl{fundamental discriminant} of~$\Delta$) and ${f=[\OO_D:\OO_\Delta]}$ is called the \textsl{conductor} of~$\Delta$.

We denote by 
$h(\Delta)$ 
the class  number of~$\OO_\Delta$. 

Given a singular modulus~$x$, we denote $\Delta_x$ the discriminant of the associated CM order, and  we write ${\Delta_x=D_xf_x^2}$ with~$D_x$ the fundamental discriminant and~$f_x$ the conductor. We denote by~$\tau_x$ the only~$\tau$ in the standard fundamental domain such that ${j(\tau)=x}$. Further, we denote by~$K_x$ the associated imaginary quadratic field
$$
K_x=\Q(\tau_x)=\Q(\sqrt{D_x})=\Q(\sqrt{\Delta_x}).
$$
We will call~$K_x$ the \textsl{CM-field} of the singular modulus~$x$.

\section{Complex analysis lemmas}
\label{scomplex}

In this section and in the subsequent Sections~\ref{sjprime} and~\ref{sdistance} the letters~$z$ and~$w$ will usually denote complex numbers, and we will systematically write
$$
z=x+yi, \quad w=u+vi.
$$
In particular, in these three sections~$x$ and~$y$ will denote real numbers, not singular moduli.

We denote by ${[z,w]}$ the straight line segment connecting ${z,w\in \C}$, that is
$$
[z,w]=\{zt+w(1-t):t\in [0,1]\}.
$$

The 
lemmas from this section are rather standard, and some of them already well known,
but we prefer to give complete proofs or exact references for the reader's convenience.

\begin{lemma}
\label{llagrange}
Let ${z,w\in \C}$ and let~$f$ be a holomorphic function on a neighborhood of  ${[z,w]}$.
Then
$$
|f(z)-f(w)|\le |z-w|\max\{|f'(\xi)|:\xi\in [z,w]\}.
$$
\end{lemma}

\begin{proof}
Consider the  function ${g(t)= f(zt+w(1-t))}$ on the interval $[0,1]$. We have
$$
|g(1)-g(0)|=\left|\int_0^1g'(t)dt\right|\le \max\{|g'(t)|:t\in [0,1]\}.
$$
Since ${g(1)-g(0)=f(z)-f(w)}$ and ${g'(t)=f'(zt+w(1-t))(z-w)}$, the result follows.
\end{proof}

This lemma gives an upper estimate for the difference ${|f(z)-f(w)|}$ in terms of ${|z-w|}$. 
For the lower estimate we will use the following lemma. 

\begin{lemma}[{\cite[Lemma~2.4]{BLP16}}]
\label{lschw}
Let~$f$ be a holomorphic function in an open neighborhood of the disc ${|z-a|\le R}$ and assume that ${|f(z)|\le B}$ in this disc. Further, let~$\ell$ be the order of vanishing of~$f$ at~$a$; that is,  
$$
f^{(k)}(a)=0 \quad (k=0,1,\ldots, \ell-1), \qquad f^{(\ell)}(a)\ne 0.
$$
Set ${A=f^{(\ell)}(a)/\ell!}$\,. Then, in the same disc  we have the estimate
\begin{equation*}
\left|f(z)-A(z-a)^\ell\right|\le \frac{|A|R^\ell+B}{R^{\ell+1}}|z-a|^{\ell+1}.
\end{equation*}
\end{lemma}

We will also need the following explicit version of the Inverse Function Theorem.

\begin{lemma}
\label{lrouche}
Let~$f$ be a holomorphic function in an open neighborhood of the disc ${|z-a|\le R}$ and assume that ${|f(z)-f(a)|\le B}$ in this disc. Furthermore, assume  that ${f'(a)=A\ne 0}$. Set
$$
C= \frac{|A|}R+\frac{B}{R^2},
$$
and let~$r$ be a positive number satisfying
\begin{equation}
\label{eassumr}
r\le\min\left\{R,
\frac{|A|}{3C}\right\}.
\end{equation}
Then, for every ${w\in \C}$ satisfying
\begin{equation}
\label{ewclose}
|w-f(a)|\le \frac{|A|}{2} r,
\end{equation}
there exists a unique~$z$ in the disc ${|z-a|\le r}$ such that ${f(z)=w}$.
\end{lemma}

\begin{proof}
By  
Lemma~\ref{lschw} applied to
the function ${f(z)-f(a)}$, 
in the disc ${|z-a |\le R}$ we have
$$
|f(z)-f(a)-A(z-a)|\le C|z-a|^2.
$$
Then on the circle ${|z-a|=r}$ we have
\begin{align*}
|f(z)-w-(f(a)+A(z-a)-w)|&\le Cr^2,\\
|f(a)+A(z-a)-w|&\ge \frac{|A|}{2}r,
\end{align*}
where in the second inequality we used~\eqref{ewclose}.

By~\eqref{eassumr}, it follows that
$$
|f(z)-w-(f(a)+A(z-a)-w)|<|f(a)+A(z-a)-w|
$$
on the circle ${|z-a|=r}$. Since the equation ${f(a)+A(z-a)=w}$ has exactly one solution in~$z$, and this solution lies in the disc ${|z-a|\le r}$ by~\eqref{ewclose}, the Theorem of Rouché implies that the equation ${f(z)=w}$ also has a unique solution in the same disc.
\end{proof}

\begin{lemma}
\label{lexp}
\begin{enumerate}
\item
For every ${z\in \C}$ satisfying ${|z|<1}$ we have
\begin{equation}
\label{eexpsmall}
|e^z-1|\ge |z|\frac{1-|z|}{1-|z|/2}.
\end{equation}
In particular, if ${|z|\le 1/2}$ then
\begin{equation}
\label{eexphalf}
|e^z-1|\ge \frac23|z|.
\end{equation}
\item
For every ${z=x+yi\in \C}$ satisfying
\begin{equation}
\label{eexpconds}
x \le 0,\quad |y|\le \pi, \quad |z|\ge 1/2,
\end{equation}
we have
\begin{equation}
\label{eexplarge}
|e^z-1|> 0.27.
\end{equation}
\end{enumerate}
\end{lemma}

\begin{proof}
We have
\begin{align*}
|e^z-1|\ge |z|-\sum_{k=2}^\infty \frac{|z|^k}{k!}\ge |z|-|z|\sum_{k=1}^\infty\left(\frac{|z|}2\right)^k=|z|\frac{1-|z|}{1-|z|/2}.
\end{align*}
This proves~\eqref{eexpsmall} for ${|z|<1}$, and~\eqref{eexphalf} for ${|z|\le 1/2}$ is an immediate consequence.

Now assume that~$z$ satisfies~\eqref{eexpconds}. If ${x\le -0.32}$ then
$$
|e^z-1|\ge 1-e^{-0.32}>0.27.
$$
Now assume that ${-0.32\le x\le 0}$. Then
$$
\pi\ge|y|\ge \sqrt{0.5^2-0.32^2}\ge 0.384.
$$
We obtain
\begin{align*}
|e^{x+iy}-1|\ge e^x|e^{iy}-1|= 2e^x\sin(y/2)\ge 2e^{-0.32}\sin(0.192)>0.27.
\end{align*}
The lemma is proved.
\end{proof}

\section{Estimates for the $j$-invariant and its derivative}
\label{sjprime}

We denote by $\calH$ the Poincaré plane, and by $\calF$ the standard fundamental domain. To be precise,~$\calF$ is the open hyperbolic triangle with vertices~$\zeta_3$,~$\zeta_6$ and~$i\infty$, together with the ``right'' part of its boundary, that is, the geodesics $[i,\zeta_6]$ and $[\zeta_6,i\infty]$.
Here
$$
\zeta_3=e^{2\pi i/3}= \frac{-1+\sqrt{-3}}{2}, \qquad \zeta_6=e^{\pi i/3}= \frac{1+\sqrt{-3}}{2}.
$$

For ${z\in \calH}$ we write ${q_z= e^{2\pi i z}}$. When there is no risk of confusion, we omit the index and write simply~$q$ instead of~$q_z$.
Recall that
\begin{equation}
\label{ejexpansion}
j(z)=\sum_{n=-1}^\infty c_nq^n,
\end{equation}
where the coefficients
$$
c_{-1}=1, \quad c_0=744,  \quad c_1=196884, \quad c_2=21493760, \quad c_3=864299970,\ \ldots
$$
are positive integers. We also denote
$$
j_0(z):=j(z)-q^{-1}-c_0= c_1q+c_2q^2+\ldots
$$

\subsection{Simple estimates}

We will systematically use the following (almost trivial) observations.

\begin{lemma}
\label{ltriv}
Let  ${z\in \calH}$ and ${v\in \R}$ be such that ${\Im z \ge v}$. Then 
\begin{align}
\label{ejm}
|j(z)-744-q_z^{-1}|=|j_0(z)|&\le j_0(iv),\\
\label{ejprimem}
|j'(z)+2\pi i q_z^{-1}| =|j_0'(z)|&\le \frac1i j_0'(iv),\\
\label{ejprimesimplelower}
|j'(z)|&\ge ij'(iv).
\end{align}
In particular, writing ${z=x+yi}$, we have 
\begin{align}
\label{ejloose}
|j(z)|&\le j(iy),\\
\label{ejprimeloose}
|j'(z)|&\le 2\pi \left(2e^{2\pi y} +\frac{1}{2\pi i} j'(iy)\right),\\
\label{ejprimelooselower}
|j'(z)|&\ge ij'(yi).
\end{align}
\end{lemma}

\begin{proof}
Set ${w=iv}$. Then ${|q_z|\le q_w}$. Since the coefficients $c_n$ of the expansion~\eqref{ejexpansion} are all positive, we have
$$
|j_0(z)|\le\sum_{n=1}^\infty c_n|q_z|^n\le \sum_{n=1}^\infty c_nq_w^n=j_0(w),
$$
which proves~\eqref{ejm}. Similarly, using that
$$
j_0'(z)=2\pi i\sum_{n=1}^\infty  nc_nq_z^n,
$$
we obtain
$$
|j_0'(z)|\le 2\pi \sum_{n=1}^\infty  nc_nq_w^n =\frac1i j_0'(w),
$$
proving~\eqref{ejprimem}.
Setting ${v=y}$ in~\eqref{ejm} and~\eqref{ejprimem}, we obtain~\eqref{ejloose} and~\eqref{ejprimeloose}.

We are left with proving~\eqref{ejprimesimplelower} and~\eqref{ejprimelooselower}.
The real function
$$
v\mapsto ij'(vi)= 2\pi \left(e^{2\pi v}-\sum_{n=1}^\infty nc_ne^{-2\pi nv}\right)
$$
is increasing in~$v$. Hence it suffices to prove~\eqref{ejprimelooselower}.  We have
\begin{align*}
|j'(z)|= 2\pi \left |-q_z^{-1}+\sum_{n=1}^\infty nc_nq_z^n\right| \ge 2\pi\left(|q_z|^{-1} -\sum_{n=1}^\infty nc_n|q_z|^n \right)= ij'(iy),
\end{align*}
as wanted.
 \end{proof}

\begin{remark}
Estimates~\eqref{ejprimesimplelower} and~\eqref{ejprimelooselower} are of interest only when ${y\ge v> 1}$, because when  ${v\le 1}$ the right-hand side of~\eqref{ejprimesimplelower} is non-positive, as well as the right-hand side of~\eqref{ejprimelooselower} when ${y \le 1}$.
\end{remark}

Let us consider the functions ${f,g:\R_{>0}\to\R_{>0}}$ defined by
\begin{equation}
\label{edeffg}
f(y)= j(iy), \quad g(y)=  2\pi \left(e^{2\pi y} + \sum_{n=1}^\infty nc_ne^{-2\pi n y}\right) =4\pi e^{2\pi y}+\frac1{ i}j'(iy)
\end{equation}
Note that the right-hand side of~\eqref{ejloose} is $f(y)$ and that   of~\eqref{ejprimeloose} is ${g(y)}$.

\begin{proposition}
\begin{enumerate}
\item
\label{if}
The function~$f$ is decreasing on $(0,1]$, increasing on $[1,+\infty)$ and satisfies ${f(1/y)=f(y)}$.
\item
\label{ig}
There exists ${y_0\in [1.018,1.019]}$ such that~$g$ is decreasing on $(0, y_0]$ and increasing on $[y_0, +\infty)$.
\end{enumerate}
\end{proposition}

\begin{proof}
Item~\ref{if} follows
from an easy computation. To prove item~\ref{ig}, we compute the derivative of~$g$, that is, 
$$
g'(y) =(2\pi)^2 \left(e^{2\pi y}-\sum_{n=1}^\infty n^2c_ne^{-2\pi ny}\right).
$$
Since the function  ${y\mapsto e^{2\pi y}}$ is increasing on~$\R$ and ${y\mapsto\sum_{n=1}^\infty n^2c_ne^{-2\pi ny}}$ is decreasing on~$\R$, the derivative vanishes at exactly one point ${y_0\in \R}$, being negative on the left of~$y_0$ and positive on the right.  A direct calculation shows that
$$
g'(1.018)=-259.31\dots, \qquad g'(1.019)= 118.15\dots,
$$
whence the result.
\end{proof}

\begin{corollary}
\label{cobetween}
Let~$\alpha,\beta$ be positive real numbers and~$\calD$  a domain in~$\calH$ such that for any ${z\in \calD}$ we have
$$
\alpha\le \Im z \le \beta.
$$
Then, for every ${z\in \cal D}$ we have
\begin{align*}
|j(z)|\le \max\{f(\alpha),f(\beta)\},\qquad
|j'(z)|\le \max\{g(\alpha),g(\beta)\}.
\end{align*}
\end{corollary}

\begin{proof}
For ${z=x+yi\in \cal D}$ we have ${|j(z)|\le f(y)}$ by~\eqref{ejloose} (recall that  
the right-hand side of~\eqref{ejloose} is $f(y)$). By the hypothesis, ${\alpha \le y\le \beta}$. Item~\ref{if} implies that ${f(y)\le f(\alpha)}$ when ${y\le 1}$ and ${f(y)\le f(\beta)}$ when ${y\ge 1}$. This proves the bound for $|j(z)|$. The proof for  $|j'(z)|$ is the same, with~$f$ replaced by~$g$ and~$1$ by~$y_0$. 
\end{proof}

\subsection{Neighborhoods of elliptic points}

Next, we want to estimate $j(z)$ and $j'(z)$ when~$z$ is close to one of the elliptic points~$\zeta_3$,~$\zeta_6$ and~$i$. Since ${\zeta_3=\zeta_6-1}$,  we restrict ourselves  to~$\zeta_6$ and~$i$.

Let us introduce the following quantities:
$$
A_0:=\frac{j'''(\zeta_6)}{3!}= -27\frac{\Gamma(1/3)^{18}}{\pi^9}i,\qquad
A_1:=\frac{j''(i)}{2!}= -81\frac{\Gamma(1/4)^8}{\pi^4}.
$$
For the calculation of the exact values of $j'''(\zeta_6)$ and $j''(i)$ see, for instance,~\cite[page~777]{Ku14}. The numerical values are
$$
|A_0|=45745.0806\ldots, \qquad |A_1|=24827.5650\ldots
$$

\begin{proposition}
\label{prinangles}
\begin{enumerate}
\item
For ${0<R<\sqrt3/2}$, set
\begin{align*}
\kappa_0(R) &:=\frac{|A_0|}{R}+\frac{f(\sqrt3/2-R)}{R^4}, \\
\lambda_0(R)&:=\frac{3|A_0|}{R}+\frac{\max\{g(\sqrt3/2-R),g(\sqrt3/2+R)\}}{R^3},
\end{align*}
where~$f$ and~$g$ are defined in~\eqref{edeffg}.
Then, in the circle ${|z-\zeta_6|\le R}$, we have 
\begin{align}
\label{ejzeta}
\bigl|j(z)-A_0(z-\zeta_6)^3\bigr|&\le \kappa_0(R)|z-\zeta_6|^4,\\
\label{ejprimezeta}
\bigl|j'(z)-3A_0(z-\zeta_6)^2\bigr|&\le \lambda_0(R)|z-\zeta_6|^3.
\end{align}

\item
For ${0<R<1}$, set
\begin{align*}
\kappa_1(R) &:=\frac{|A_1|}{R}+\frac{f(1-R)}{R^3}, \\
\lambda_1(R)&:=\frac{2|A_1|}{R}+\frac{\max\{g(1-R),g(1+R)\}}{R^2}.
\end{align*}
Then, in the circle ${|z-i|\le R}$, we have 
\begin{align}
\label{eji}
\bigl|j(z)-1728-A_1(z-i)^2\bigr|&\le \kappa_1(R)|z-i|^3,\\
\label{ejprimei}
\bigl|j'(z)-2A_1(z-i)\bigr|&\le \lambda_1(R)|z-i|^2.
\end{align}
\end{enumerate}
\end{proposition}

\begin{proof}
Corollary~\ref{cobetween} implies that in the circle ${|z-\zeta_6|\le R}$
\begin{align*}
|j(z)|&\le \max\{f(\sqrt3/2-R),f(\sqrt3/2+R)\},\\
|j'(z)|&\le \max\{g(\sqrt3/2-R),g(\sqrt3/2+R)\}
\end{align*}
Since
$$
(\sqrt3/2-R)(\sqrt3/2+R)<1,
$$
we have
$$
\max\{f(\sqrt3/2-R),f(\sqrt3/2+R)\}= f(\sqrt3/2-R).
$$
Now applying Lemma~\ref{lschw} we obtain~\eqref{ejzeta} and~\eqref{ejprimezeta}. The proof of~\eqref{eji} and~\eqref{ejprimei} is analogous .
\end{proof}

\begin{corollary}
\label{coinangles}
For ${|z-\zeta_6|\le 0.19}$ we have
\begin{align}
\label{ejzetanum}
\bigl|j(z)-A_0(z-\zeta_6)^3\bigr|&\le 7.26\cdot10^6|z-\zeta_6|^4,\\
\label{ejprimezetanum}
\bigl|j'(z)-3A_0(z-\zeta_6)^2\bigr|&\le 2.27\cdot10^7|z-\zeta_6|^3.
\end{align}
For ${|z-i|\le 0.2}$ we have
\begin{align}
\label{ejinum}
\bigl|j(z)-1728-A_1(z-i)^2\bigr|&\le 4.04\cdot10^5|z-i|^3,\\
\label{ejprimeinum}
\bigl|j'(z)-2A_1(z-i)\bigr|&\le 9.1\cdot10^5|z-i|^2.
\end{align}
\end{corollary}

\begin{proof}
Set the ``quasi-optimal'' values ${R=0.25}$ in~\eqref{ejzeta}, ${R=0.19}$ in~\eqref{ejprimezeta}, ${R=0.29}$ in~\eqref{eji} and ${R=0.2}$ in~\eqref{ejprimei}.
\end{proof}

\subsection{Global lower estimates}
\label{ssglo}

Using  Corollary~\ref{coinangles}, we
obtain the following lower estimates. 

\begin{proposition} Let~$z$ belong to~$\calF$. 
\label{pglobalj}
\begin{enumerate}
\item
\label{izeta6}
We have one of the following two options: either
\begin{equation}
\label{ejnearzetasixorthree}
\text{$\min\{|z-\zeta_6|,|z-\zeta_3|\}\le 0.001$ and $|j(z)|\ge 30000\min\{|z-\zeta_6|,|z-\zeta_3|\}^3$},
\end{equation}
or
\begin{equation}
\label{ejfarfromzetasixandthree}
\text{$\min\{|z-\zeta_6|,|z-\zeta_3|\}\ge 0.001$  and $|j(z)|\ge 3\cdot10^{-5}$}.
\end{equation}

\item
\label{ii}
We have one of the following two options: either
\begin{equation}
\label{ejneari}
\text{$|z-i|\le 0.01$ and $|j(z)-1728|\ge  20000|z-i|^2$},
\end{equation}
or
\begin{equation}
\label{ejfarfromi}
\text{$|z-i|\ge 0.01$ and $|j(z)-1728|\ge 2$}.
\end{equation}
\end{enumerate}
\end{proposition}

\begin{proof}
When ${|z-\zeta_6|\le 0.001}$ we have
$$
|j(z)|\ge (A_0-7.26\cdot 10^6\cdot 0.001)|z-\zeta_6|^3>30000|z-\zeta_6|^3.
$$
Similarly, when ${|z-\zeta_3|\le 0.001}$ we have ${|j(z)|>30000|z-\zeta_3|^3}$.
In particular, if ${|z-\zeta_6|= 0.001}$  or ${|z-\zeta_3|= 0.001}$
then
$$
|j(z)|\ge 30000\cdot(0.001)^3=3\cdot10^{-5}.
$$
From the known behavior of~$j$ on the boundary of~$\calF$ we conclude that the estimate ${|j(z)|\ge 3\cdot10^{-5}}$ holds for every~$z$ on the boundary of the set
\begin{equation}
\label{efwithoutears}
\{z\in \calF: \text{ ${|z-\zeta_6|\ge 0.001}$ and ${|z-\zeta_3|\ge 0.001}$}\}.
\end{equation}
Since~$j$ does not vanish on the set~\eqref{efwithoutears}, the maximum principle implies that ${|j(z)|\ge 3\cdot10^{-5}}$ for every~$z$ in the set~\eqref{efwithoutears}. This proves item~\ref{izeta6}.

The proof of item~\ref{ii} is analogous.
\end{proof}

Unfortunately, we cannot apply the same argument to~$j'$, because we do not have enough information on its behavior on the boundary of~$\calF$. However, this can be overcome using the following simple lemma. We use the classical notation
\begin{align*}
E_{2k}(z)= 1- \frac{4k}{B_{2k}}\sum_{n=1}^\infty \sigma_{2k-1}(n)q^n,\qquad
\Delta(z)=\frac{E_4(z)^3-E_6(z)^2}{1728},
\end{align*}
where ${\sigma_k(n)=\sum_{d\mid n}d^k}$ and~$B_k$ is the $k$th Bernoulli number. 

Note that here (and until the end of Section~\ref{ssglo}) the letter~$\Delta$ denotes the modular form~$\Delta(z)$, and not an imaginary quadratic discriminant (as in the rest of the article).

For more details the reader may consult any introductory course on modular forms, for instance, \cite[Sections~1.1,~1.2]{DS05}. (A warning:  $\Delta(z)$ in~\cite{DS05} corresponds to ${(2\pi)^{12}\Delta(z)}$ in our notation.)


\begin{lemma}
\label{lkuehne}
For any ${z\in \calH}$ we have
\begin{equation}
\label{ekuehne}
|j'(z)|\ge2\pi\min \bigl\{|j(z)|,|\Delta(z)|^{1/3}|j(z)|^{1/3}|j(z)-1728|\bigr\}.
\end{equation}
\end{lemma}

\begin{proof}
We have
$$
j(z)=\frac{E_4(z)^3}{\Delta(z)}, \qquad j(z)-1728=\frac{E_6(z)^2}{\Delta(z)}.
$$
Furthermore,
$$
\frac1{2\pi i}j'(z)=-\frac{E_4(z)^2E_6(z)}{\Delta(z)}= - \frac{E_6(z)}{E_4(z)}j(z)= - \frac{E_4(z)^2}{E_6(z)}(j(z)-1728),
$$
see, for instance, \cite[page~775]{Ku14}.
Hence
\begin{align*}
\frac{|j'(z)|}{2\pi}&\ge \left|\frac{E_6(z)}{E_4(z)}\right||j(z)|, \\
\frac{|j'(z)|}{2\pi}&\ge\left|\frac{E_4(z)^2}{E_6(z)}\right||j(z)-1728|\\
&=\left|\frac{E_4(z)}{E_6(z)}\right||\Delta(z)|^{1/3}|j(z)|^{1/3}|j(z)-1728|.
\end{align*}
Since either ${|E_6(z)/E_4(z)|\ge 1}$ or ${|E_4(z)/E_6(z)|\ge 1}$, the result follows.
\end{proof}

\begin{remark}
In the neighborhoods of elliptic points one expects sharper lower bounds: there must be ${|j'(z)|\gg |j(z)|^{2/3}}$ near the elliptic points of type~$\zeta_6$, and  ${|j'(z)|\gg |j(z)-1728|^{1/2}}$ near the elliptic points of type~$i$.  This can be accomplished as well, see \cite[page~777]{Ku14}. However, for our purposes~\eqref{ekuehne} will be sufficient.
\end{remark}

\begin{proposition}
\label{pglobaljprime}
Let~$z$ belong to~$\calF$.
Then we have one of the following three options: either
\begin{equation}
\label{ejprimenearzetasixorthree}
\text{$\min\{|z-\zeta_6|,|z-\zeta_3|\}\le 0.001$ and $|j'(z)|\ge 10^5\min\{|z-\zeta_6|,|z-\zeta_3|\}^2$},
\end{equation}
or
\begin{equation}
\label{ejprimeneari}
\text{$|z-i|\le 0.01$ and $|j'(z)|\ge 40000|z-i|$},
\end{equation}
or
\begin{equation}
\label{ejprimefarfromzetasixandthreeandi}
\text{$\min\{|z-\zeta_6|,|z-\zeta_6|\}\ge 0.001$,  $|z-i|\ge 0.01$ and $|j'(z)|\ge 10^{-4}$}.
\end{equation}
\end{proposition}

\begin{proof}
The cases ${\min\{|z-\zeta_6|,|z-\zeta_3|\}\le 0.001}$ and ${|z-i|\le 0.01}$ are treated exactly as in the proof of Proposition~\ref{pglobalj}, using the corresponding instances of Corollary~\ref{coinangles}. When ${\Im z\ge 1.01}$ estimate~\eqref{ejprimesimplelower} gives
${|j'(z)|\ge ij'(1.01i)\ge 400}$,
which is much sharper than the wanted  ${|j'(z)|\ge 10^{-4}}$.

We are left with proving that ${|j'(z)|\ge 10^{-4}}$ in the case
$$
\min\{|z-\zeta_6|,|z-\zeta_6|\}\ge 0.001, \qquad |z-i|\ge 0.01, \qquad\Im z\le 1.01.
$$
Proposition~\ref{pglobalj} gives
\begin{equation}
\label{ecorprop}
|j(z)|\ge 3\cdot10^{-5}, \qquad |j(z)-1728|\ge 2.
\end{equation}
We want to apply Lemma~\ref{lkuehne}, and for this purpose we need a lower bound for $|\Delta(z)|^{1/3}$. This can be easily accomplished using the classical infinite product expansion
${\Delta(z)=q\prod_{n=1}^\infty(1-q^n)^{24}}$.
Using the inequality
$$
\log|1+t|\ge -\frac{|t|}{1-|t|}
$$
which holds true for all complex~$t$ satisfying ${|t|<1}$, we obtain
\begin{align*}
\log|\Delta(z)|^{1/3}\ge \frac13\log |q| -8\sum_{n=1}^\infty \frac{|q|^n}{1-|q|}=  \frac13\log |q| -8 \frac{|q|}{(1-|q|)^2}.
\end{align*}
Since ${z\in \calF}$ and ${\Im z \le 1.01}$,  we have ${e^{-2.02\pi}\le |q|\le e^{-\pi\sqrt3}}$, which gives the lower estimate
\begin{equation}
\label{ebigdeltalower}
\log|\Delta(z)|^{1/3}\ge -\frac{2.02\pi}3  -8 \frac{e^{-\pi\sqrt3}}{(1-e^{-\pi\sqrt3})^2}\ge -2.16.
\end{equation}

Now we are ready to apply Lemma~\ref{lkuehne}. Combining it with~\eqref{ecorprop} and~\eqref{ebigdeltalower}, we obtain
\begin{align*}
|j'(z)|\ge2\pi\min \bigl\{3\cdot10^{-5},e^{-2.16}\cdot(3\cdot10^{-5})^{1/3}\cdot2\bigr\} \ge 10^{-4},
\end{align*}
as wanted.
\end{proof}

\section{Separating distinct $j$-values}
\label{sdistance}
In this section we bound from below the difference ${|j(z)-j(w)|}$, where~$z$ and~$w$ are distinct elements of the fundamental domain~$\calF$.

\begin{proposition}
\label{pbigimwcorrected}
Let ${z,w\in \calF}$ satisfy
${\Im w \ge \Im z}$ and ${\Im w \ge 1.3}$.
Then, there exists ${z'\in \{z,z+1,z-1\}}$ such that
\begin{equation}
\label{elower12}
|j(z)-j(w)|\ge e^{2\pi\Im w}\min \{0.4|z'-w|,0.05\}.
\end{equation}
\end{proposition}

\begin{proof}
We have
\begin{equation}
\label{emaindif}
|j(z)-j(w)|\ge \bigl|q_w^{-1}-q_z^{-1}\bigr|-|j_0(w)-j_0(z)|.
\end{equation}
Assume first that ${\Im z \le 1.2}$. In this case
\begin{equation*}
|q_w^{-1}-q_z^{-1}| \ge e^{2\pi v}(1-e^{-2\pi\cdot 0.1}) \ge 0.45e^{2\pi v},
\end{equation*}
where ${v=\Im w}$. On the other hand, since ${z\in \calF}$, we have ${\Im z \ge \sqrt3/2}$. Using~\eqref{ejm} and our assumption ${v\ge 1.3}$, we find
\begin{equation*}
|j_0(w)-j_0(z)|\le j_0\left(\frac{\sqrt3}2i\right)+j_0(1.3i)  \le 1400 \le 0.4e^{2\pi v}.
\end{equation*}
Substituting  these two estimates to~\eqref{emaindif}, we deduce that
$$
|j(z)-j(w)|\ge  0.05e^{2\pi v}, 
$$
which completes the proof in the case ${\Im z \le 1.2}$.

\medskip

From now on we assume that ${\Im z \ge 1.2}$. Let us first estimate from above the difference ${|j_0(w)-j_0(z)|}$. Under the condition ${\Im z\ge 1.2}$, we have
$$
|j_0'(z)|\le \frac1i j_0'(1.2i) <800,
$$
see~\eqref{ejprimem}. Hence
\begin{equation*}
|j_0(w)-j_0(z)|\le 800|w-z|,
\end{equation*}
see Lemma~\ref{llagrange}. Since ${j_0(z\pm1)=j_0(z)}$, we may replace here~$z$ by ${z\pm1}$ and obtain similar inequalities with ${|w-(z\pm1)|}$ on the right. This proves that
\begin{equation*}
|j_0(w)-j_0(z)|\le 800|w-z'|
\end{equation*}
for every ${z'\in \{z,z-1,z+1\}}$.
In addition to this, using~\eqref{ejm}, we find
\begin{equation*}
|j_0(w)-j_0(z)|\le j_0(1.3i)+j_0(1.2i) <200.
\end{equation*}
Since ${v\ge 1.3}$, we proved that 
\begin{align}
|j_0(w)-j_0(z)| &\le \min \{200, 800|w-z'|\}\nonumber\\
\label{edifmin}
&\le e^{2\pi v}\min\{0.06, 0.23|w-z'|\}
\end{align}
for every ${z'\in \{z,z-1,z+1\}}$.

Now let us estimate the difference ${|q_z^{-1}-q_w^{-1}|}$ from below. There exists a unique ${z'\in \{z,z-1,z+1\}}$ such that ${|\Re(z'-w)|\le 1/2}$, and we maintain this choice of~$z'$ in the sequel.
We have clearly
$$
|q_z^{-1}-q_w^{-1}|=e^{2\pi v}|e^{2\pi i(w-z')}-1|.
$$
Our assumption ${\Im w\ge\Im z}$ implies that
$$
\Re(2\pi i(w-z')) \le 0,
$$
and the choice of~$z'$ implies that ${|\Re (w-z')|\le 1/2}$, which can be re-written as
$$
|\Im(2\pi i(w-z'))|\le \pi.
$$
Now we want to apply Lemma~\ref{lexp} with ${2\pi i(w-z')}$ as~$z$. 
Applying~\eqref{eexphalf},~\eqref{eexplarge}, we obtain
\begin{equation}
\label{esmallcase}
|q_z^{-1}-q_w^{-1}|\ge \frac23 e^{2\pi v}|w-z'| ,
\end{equation}
if ${|w-z'|\le 1/(4\pi)}$, and
\begin{equation}
\label{ebigcase}
|q_z^{-1}-q_w^{-1}|\ge 0.27 e^{2\pi v},
\end{equation}
if ${|w-z'|\ge 1/(4\pi)}$. 
Combining these estimates with~\eqref{emaindif} and~\eqref{edifmin}, we obtain
 \begin{equation*}
|j(z)-j(w)|\ge
\begin{cases}
0.4 e^{2\pi v}|w-z'| , &|w-z'|\le 1/(4\pi),\\
0.2 e^{2\pi v}, &|w-z'|\ge 1/(4\pi),
\end{cases}
\end{equation*}
 sharper than~\eqref{elower12}.
\end{proof}

Given ${S\subset \C}$ and ${\eps>0}$, we define the  $\eps$-neighborhood of~$S$ as the set of all ${z\in \C}$ such that ${|z-w|<\eps}$ for some ${w\in S}$.

\begin{proposition}
\label{pseparjvalues}
Assume that ${z,w\in \calF}$ and ${\Im w\le 1.3}$. Then, there exists~$z'$ in the ${10^{-5}}$-neighborhood of~$\calF$ such that ${j(z')=j(z)}$ and
\begin{equation}
\label{eseparjvalues}
|j(z)-j(w)|\ge  \min\bigl\{5\cdot10^{-7},5\cdot10^{-7}|j'(w)|^2, 0.6|j'(w)||z'-w|\bigr\}.
\end{equation}
\end{proposition}

\begin{proof}
Let~$R$ be a constant (to be specified later)  satisfying ${0<R<\sqrt{3}/2}$. Set
\begin{align*}
B&:=\max\bigl\{f({\sqrt3}/2-R),f(1.3+R)\bigr\},
\\
r&:=\min \left\{R,\frac{|j'(w)|R^2}{3(B+R)}, \frac{R^2}{3(B+R)}\right\},
\end{align*}
where, as before, ${f(y)=j(iy)}$. 
Since ${\sqrt3/2\le \Im w \le 1.3}$, 
Corollary~\ref{cobetween} implies that every~$\xi$ in the disk ${|\xi-w|\le R}$ satisfies ${|j(\xi)|\le B}$.

We will now use Lemma~\ref{lrouche} with~$j$ as~$f$, with~$w$ as~$a$ and with~$j(z)$ as~$w$. Condition~\eqref{eassumr} translates into
\begin{equation}
\label{eassumrverified}
r\le \min \left\{R, \frac{|j'(w)|}{3(|j'(w)|/R+B/R^2)}\right\}.
\end{equation}
We have clearly
$$
\frac{|j'(w)|}{3(|j'(w)|/R+B/R^2)} \ge \frac{\min\{|j'(w)|,1\}R^2}{3(B+R)}.
$$
Hence~\eqref{eassumrverified} holds true by our definition of~$r$.

Lemma~\ref{lrouche} implies  that there are two possibilities: either
\begin{equation}
\label{ebigdistance}
|j(z)-j(w)|\ge \frac12r|j'(w)|,
\end{equation}
or there exists ${z'\in  \calH}$ such that ${j(z')=j(z)}$ and ${|z'-w|\le r}$. In the latter case
Lemma~\ref{lschw} implies that
\begin{align*}
|j(z')-j(w)-j'(w)(z'-w)|&\le (|j'(w)|/R+B/R^2)|z'-w|^2 \\
&\le (|j'(w)|/R+B/R^2)r|z'-w|.
\end{align*}
Using~\eqref{eassumrverified}, we find that
$$
(|j'(w)|/R+B/R^2)r \le \frac13|j'(w)|,
$$
which implies that
\begin{equation}
\label{esmalldistance}
|j(z')-j(w)|\ge \frac23|j'(w)||z'-w|.
\end{equation}

Thus, we have either~\eqref{ebigdistance} or~\eqref{esmalldistance}. Setting a ``nearly optimal''  ${R=0.2}$, we obtain
$$
10^{-5}>\frac{R^2}{3(B+R)} > 10^{-6};
$$
in particular,
$$
10^{-5}>r>10^{-6}\min \{1,|j'(w)|\}.
$$
Hence~\eqref{ebigdistance} implies that
\begin{equation}
\label{ebigdfinal}
|j(z)-j(w)|\ge 5\cdot10^{-7}\min \{1,|j'(w)|^2\}.
\end{equation}
Thus, we have either~\eqref{ebigdfinal} or~\eqref{esmalldistance}, which proves~\eqref{eseparjvalues} with  our choice of~$z'$. Finally, we note that ${|z'-w|\le r<10^{-5}}$, which shows that~$z'$ belongs to the $10^{-5}$-neighborhood of~$\calF$.
\end{proof}


\section{Separating imaginary quadratic numbers}
\label{ssimag}
Call a complex number \textsl{imaginary quadratic}  if it is algebraic of degree~$2$ over~$\Q$ and does not belong to~$\R$. By the \textsl{discriminant} of an imaginary quadratic number we mean the discriminant of its minimal polynomial over~$\Z$. 

We want to bound from below the distance between two imaginary quadratic numbers. Of course, this can be done using the ``Liouville inequality'': if~$\alpha$ and~$\alpha'$ are distinct complex algebraic numbers then ${|\alpha-\alpha'|\ge \bigl(2H(\alpha)H(\alpha')\bigr)^{-d}}$, where $H(\cdot)$ is the absolute (multiplicative) height and ${d=[\Q(\alpha,\alpha'):\Q]}$. However, for imaginary quadratic numbers finer bounds can be proved.

\begin{proposition}
\label{pseparquad}
Let $\alpha,\alpha'$ be two distinct imaginary quadratic numbers with positive imaginary parts, and let $\Delta,\Delta'$ be their respective discriminants. Then
\begin{equation}
\label{esepquar}
|\alpha-\alpha'|\ge
\begin{cases}
\frac{2\Im\alpha\Im\alpha'\min\{\Im\alpha,\Im\alpha'\}}{|\Delta||\Delta'|}, & \text{if}\quad  \Im\alpha \ne \Im\alpha',\\
\frac{\Im\alpha}{|\Delta|^{1/4}|\Delta'|^{1/4}}, & \text{if}\quad  \Im\alpha = \Im\alpha'.
\end{cases}
\end{equation}
\end{proposition}

\begin{proof}
Let ${aX^2+bX+c\in \Z[X]}$ be the minimal polynomial of~$\alpha$ over~$\C$ with ${a>0}$. Then
$$
\Delta=b^2-4ac, \qquad \alpha=\frac{-b+|\Delta|^{1/2}i}{2a} =\beta+\delta i,
$$
with
$$
\beta=-\frac b{2a}=\Re \alpha, \qquad \delta=\frac{|\Delta|^{1/2}}{2a}=\Im \alpha.
$$
Similarly, if ${a'X^2+b'X+c'\in \Z[X]}$ is the minimal polynomial of~$\alpha'$ over~$\Z$, then 
$$
\alpha'=\frac{-b'+|\Delta'|^{1/2}i}{2a'} =\beta'+\delta' i.
$$

If ${\delta\ne \delta'}$ then
\begin{align*}
|\alpha-\alpha'|&\ge |\delta-\delta'|\\
&= \frac{\bigl|(a')^2|\Delta|-a^2|\Delta'|\bigr|}{2aa'(a'|\Delta|^{1/2}+a|\Delta'|^{1/2})}\\
&\ge \frac{1}{2aa'(a'|\Delta|^{1/2}+a|\Delta'|^{1/2})}\\
&= \frac{4(\delta\delta')^2}{|\Delta||\Delta'|(\delta+\delta')}\\
&\ge \frac{2\delta\delta'\min\{\delta,\delta'\}}{|\Delta||\Delta'|},
\end{align*}
which proves~\eqref{esepquar} in the case ${\Im\alpha\ne \Im\alpha'}$.

Now assume that ${\delta =\delta'}$. In this case~$\alpha$ and~$\alpha'$ generate the same imaginary quadratic field:
$$
\Q(\alpha)=\Q(\sqrt\Delta)=\Q(\alpha')=\Q(\sqrt{\Delta'}).
$$
Denote by~$D$ the discriminant of this field. Then ${\Delta=Df^2}$, ${\Delta'=D(f')^2}$ with some positive integers~$f$ and~$f'$.

Denote ${e=\gcd(f,f')}$.
Since ${\delta =\delta'}$, we have ${af'=a'f}$, and
$$
a=A\frac fe, \qquad a'=A \frac{f'}e, \qquad \delta=\delta'= \frac{e|D|^{1/2}}{2A}
$$
with some ${A\in \Z}$. Furthermore, the relation ${\Delta=b^2-4ac}$ implies that ${(f/e)\mid b^2}$. Hence
${b/(f/e)=b_1/f_1}$, where $b_1,f_1$ are  integers such that ${0<f_1\le (f/e)^{1/2}}$. Similarly, ${b'/(f'/e)=b'_1/f'_1}$, where ${0<f'_1\le (f'/e)^{1/2}}$. Using all this, we obtain 
\begin{align*}
|\alpha-\alpha'|&=|\beta-\beta'|
=\frac{|b_1f'_1-b'_1f_1|}{2Af_1f'_1}\ge \frac{1}{2A(f/e)^{1/2}(f'/e)^{1/2}}= \frac{\delta}{|\Delta|^{1/4}|\Delta'|^{1/4}},
\end{align*}
which proves~\eqref{esepquar} in the case ${\Im\alpha= \Im\alpha'}$.
\end{proof}

\begin{remark}
If ${\Delta=\Delta'}$, then we have the sharper estimate
\begin{equation}
\label{esepquarconj}
|\alpha-\alpha'|\ge
\begin{cases}
\frac{\Im\alpha\Im\alpha'}{2|\Delta|^{1/2}}, & \text{if}\quad \Im\alpha \ne \Im\alpha',\\
\frac{\Im\alpha}{|\Delta|^{1/2}}, & \text{if}\quad  \Im\alpha = \Im\alpha'.
\end{cases}
\end{equation}
Indeed, if ${\delta=\delta'}$ then~\eqref{esepquarconj} follows from~\eqref{esepquar}. On the other hand, if ${\delta\ne \delta'}$, then ${a\ne a'}$, and we obtain
$$
|\alpha-\alpha'|\ge |\delta-\delta'|= \frac{|a-a'||\Delta|^{1/2}}{2aa'}\ge \frac{|\Delta|^{1/2}}{2aa'}= \frac{|\delta||\delta'|}{2|\Delta|^{1/2}}.
$$
Unfortunately, we were not able to profit from~\eqref{esepquarconj} to refine Theorem~\ref{thseparweak} in the (apparently, most important) special case ${\Delta_x=\Delta_y}$. The reason is that in the proof of Theorem~\ref{thsepar} below, we are going to combine Proposition~\ref{pseparquad} with the lower bound~\eqref{eseparjvalues}, with~$w$ and~$z'$ imaginary quadratic. This bound involves not only the term ${|j'w)||z'-w|}$,  which indeed can be refined by refining the lower bound for ${|z'-w|}$, but also the term $|j'(w)|^2$, for which the lower bound for ${|z'-w|}$ is irrelevant.  
\end{remark}

\begin{corollary}
\label{cderiv}
Let~$\tau$ be an imaginary quadratic number of discriminant~$\Delta$. Assume that ${\tau\in\calF}$ and ${\tau\ne i,\zeta_6}$.  Then
\begin{align}
\label{eseparfromzeta}
\min \{|\tau-\zeta_6|,|\tau-\zeta_3|\}& \ge \frac{\sqrt3}{4}|\Delta|^{-1},\\
\label{eseparfromi}
|\tau-i|&\ge \frac{3}{8}\Delta|^{-1},\\
\label{ejseparfromzero}
|j(\tau)|&\ge 700|\Delta|^{-3},\\
\label{ejseparfromtwelvecube}
|j(\tau)-1728|&\ge 2000|\Delta|^{-2},\\
\label{ederiv}
|j'(\tau)|&\ge 40000|\Delta|^{-2}.
\end{align}
\end{corollary}

\begin{proof}
Estimates~\eqref{eseparfromzeta} and~\eqref{eseparfromi} are obtained using Proposition~\ref{pseparquad} with ${\alpha=\tau}$ and ${\alpha'=\zeta_6,\zeta_3,i}$, respectively; note that ${\Im\tau\ge \sqrt3/2}$ because ${\tau\in \calF}$.

To obtain~\eqref{ejseparfromzero},~\eqref{ejseparfromtwelvecube} and~\eqref{ederiv} we combine Propositions~\ref{pglobalj} and~\ref{pglobaljprime} with estimates~\eqref{eseparfromzeta} and~\eqref{eseparfromi}. We obtain
\begin{align}
|j(\tau)|&\ge \min\{3\cdot10^{-5},700|\Delta|^{-3}\},\nonumber\\
\label{ejseparfromtwelvecubewithconstant}
|j(\tau)-1728|&\ge \min\{2,2000|\Delta|^{-2}\},\\
\label{ederivythree}
|j'(\tau)|&\ge \min \{10^{-4},15000|\Delta|^{-1}, 40000|\Delta|^{-2}\}.
\end{align}
Note that ${3\cdot10^{-5}>700|\Delta|^{-3}}$ when ${|\Delta|\ge300}$. A quick \textsf{PARI} script shows that ${|j(\tau)|\ge700|\Delta|^{-3}}$ when~$\tau$ has discriminant~$\Delta$ satisfying ${|\Delta|\le 300}$. 
This proves inequality~\eqref{ejseparfromzero}.

Similarly, using a quick calculation with \textsf{PARI} one gets rid of~$2$ in~\eqref{ejseparfromtwelvecubewithconstant}, proving~\eqref{ejseparfromtwelvecube}.

Finally, since ${|\Delta|\ge 3}$ we have ${15000|\Delta|^{-1}\ge40000|\Delta|^{-2}}$, and~\eqref{ederivythree} becomes
$$
|j'(\tau)|\ge \min \{10^{-4}, 40000|\Delta|^{-2}\}.
$$
We have ${10^{-4}\ge 40000|\Delta|^{-2}}$ when ${|\Delta|\ge 20000}$, and we again use a  \textsf{PARI} script 
to show that ${|j'(\tau)|\ge40000|\Delta|^{-2}}$ when~$\tau$ has discriminant~$\Delta$ with ${|\Delta|\le 20000}$.
This proves~\eqref{ederiv}.
\end{proof}

\section{Separating singular moduli}
\label{sseparating}
In this section we prove the first principal result of this article. Recall that we denote by~$\Delta_x$ the fundamental discriminant of the singular modulus~$x$.

\begin{theorem}
\label{thsepar}
Let~$x,y$ be distinct singular moduli. Assume that ${|\Delta_x|\ge |\Delta_y|}$. Then
\begin{equation}
\label{esepar}
|x-y|\ge \min
\{
800|\Delta_y|^{-4}, 20000|\Delta_x|^{-1}|\Delta_y|^{-3},700|\Delta_x|^{-3}\}.
\end{equation}
\end{theorem}

\begin{proof}
Let ${\tau_x,\tau_y\in \calF}$ be such that ${j(\tau_x)=x}$ and ${j(\tau_y)=y}$. Assume first that ${\Im \tau_y\ge 1.3}$. In this case Proposition~\ref{pbigimwcorrected} implies that
\begin{equation}
\label{ebigimtaucorrected}
|x-y|\ge e^{2\pi \Im\tau_y}\min \{0.4|\tau_x'-\tau_y|,0.05\}\ge \min\{1400|\tau_x'-\tau_y|,170\}.
\end{equation}
where ${\tau_x'\in \{\tau_x,\tau_x-1,\tau_x+1\}}$. We have ${\Im\tau_x'=\Im\tau_x\ge \sqrt3/2}$ because ${\tau_x\in \calF}$. Hence, using Proposition~\ref{pseparquad}, we obtain
\begin{equation*}
|\tau_x'-\tau_y|\ge
|\Delta_x|^{-1}|\Delta_y|^{-1}.
\end{equation*}
Combining this with~\eqref{ebigimtaucorrected} we obtain an estimate much sharper than~\eqref{esepar}.

Now let us assume that ${\Im \tau_y\le 1.3}$ and ${y\ne 0,1728}$. In this case Proposition~\ref{pseparjvalues} implies that
\begin{equation}
\label{esmallimtau}
|x-y|\ge \min\bigl\{5\cdot10^{-7},5\cdot10^{-7}|j'(\tau_y)|^2, 0.6|j'(\tau_y)||\tau_x'-\tau_y|\bigr\},
\end{equation}
where~$\tau_x'$ belongs to the ${10^{-5}}$-neighborhood of~$\calF$ and ${j(\tau_x')=x}$.

We have
\begin{equation}
\label{elowtaus}
\Im\tau_y\ge \sqrt3/2, \qquad\Im\tau_x'\ge \sqrt3/2-10^{-5}.
\end{equation}
Hence, using Proposition~\ref{pseparquad}, we obtain
\begin{equation*}
|\tau_x'-\tau_y|\ge
\begin{cases}
|\Delta_x|^{-1}|\Delta_y|^{-1},& \text{if}\quad \Im\tau_x'\ne\Im\tau_y, \\
(\sqrt3/2)|\Delta_x|^{-1/4}|\Delta_x|^{-1/4},&\text{if}\quad \Im\tau_x'=\Im\tau_y.
\end{cases}
\end{equation*}
Since ${|\Delta_x|,|\Delta_y|\ge 3}$, we have ${|\Delta_x|^{3/4}|\Delta_y|^{3/4}\ge 2/\sqrt3}$, which implies that
\begin{equation}
\label{ebetweentaus}
|\tau_x'-\tau_y|\ge
|\Delta_x|^{-1}|\Delta_y|^{-1}
\end{equation}
in any case.
In addition to this, since ${y\ne 0,1728}$, we have ${\tau_y\ne\zeta_6,i}$.  Hence~\eqref{ederiv} implies that
\begin{equation*}
|j'(\tau_y)|\ge 40000|\Delta_y|^{-2}.
\end{equation*}
Combining this with~\eqref{esmallimtau} and~\eqref{ebetweentaus} we obtain
$$
|x-y|\ge \min
\{5\cdot10^{-7}, 800|\Delta_y|^{-4}, 20000|\Delta_x|^{-1}|\Delta_y|^{-3}\}.
$$
Finally, when ${y=0}$ or ${y=1728}$ Corollary~\ref{cderiv} implies that
${|x-y|\ge 700|\Delta_x|^{-3}}$.

Thus, we have proved that
\begin{equation}
\label{eseparwithconstant}
|x-y|\ge \min
\{5\cdot10^{-7}, 800|\Delta_y|^{-4}, 20000|\Delta_x|^{-1}|\Delta_y|^{-3},700|\Delta_x|^{-3}\};
\end{equation}
to conclude, we have to get rid of the term ${5\cdot10^{-7}}$ on the right.

Note that
\begin{align*}
800|\Delta_y|^{-4}&\le 5\cdot10^{-7}&& \text{when $|\Delta_y|\ge200$},\\
700|\Delta_x|^{-3}&\le 5\cdot10^{-7}&& \text{when $|\Delta_x|\ge1119$}.\
\end{align*}
Hence we have to verify that~\eqref{esepar} holds true when ${|\Delta_y|\le200}$ and ${|\Delta_x|\le1119}$. We did it  using 
a \textsf{PARI} script.
\end{proof}

For small values of discriminants, much better lower bounds hold true. Using a PARI script, we proved the following proposition, which 
will be used several times in Section~\ref{sprim}.

\begin{proposition}
\label{psmaldis}
Let $X_k$, $d_k$ and $d'_k$ be the numbers defined in Table~\ref{tasmaldis}.
\begin{table}
\caption{Data for Proposition~\ref{psmaldis}}
\label{tasmaldis}
$$
\begin{array}{c|llll}
k&1&2&3&4\\
\hline
X_k& 300&1000&3000&10000\\
d_k & 3.82&0.305&0.0292&0.00247\\
d'_k & 92.4&15.7&3.07&0.494
\end{array}
$$
\end{table}
Then, for any distinct singular moduli  $x,y$ with ${|\Delta_x|,|\Delta_y|\le X_k}$ we have ${|x-y|\ge d_k}$. Moreover, if ${\Delta_x=\Delta_y}$ then ${|x-y|\ge d'_k}$.
\end{proposition}

\section{More on singular moduli}
\label{ssing}
In this section we summarize some properties of singular moduli that will be used in the proof of Theorem~\ref{thalpha}.

\subsection{Galois-theoretic properties}
The following properties of singular moduli will be systematically used in the sequel, usually without special reference. For more details, the reader may consult, for instance,~\cite{Co13}, especially §7,~§9D,~§11 and~§13 therein.  

\begin{itemize}
\item
A singular modulus~$x$ is an algebraic number (even algebraic integer) of degree equal to the class number $h(\Delta_x)$. 

\item
Two singular moduli~$x$ and~$y$ are conjugate over~$\Q$ if and only if ${\Delta_x=\Delta_y}$. In other words, singular moduli of given discriminant~$\Delta$ form a Galois orbit over~$\Q$, of cardinality $h(\Delta)$.

\item
If ${K_x=\Q(\sqrt{\Delta_x})}$ is the CM field associated to the singular modulus~$x$ then ${K_x(x)/K_x}$ is an abelian Galois extension of degree $h(\Delta_x)$, and ${K_x(x)/\Q}$ is a Galois extension of degree $2h(\Delta)$ (in general not abelian).

\item
The algebraic extension $\Q(x)/\Q$ is usually not Galois, but if it is, it must be $2$-elementary; that is, the Galois group is of  the type ${\Z/2\Z\times \cdots\times \Z/2\Z}$. (The proof can be found, for instance, in  \cite[Corollary 3.3]{ABP15}.) In this case ${K_x(x)/\Q}$ is 2-elementary as well.  If $\Q(x)$ is  not Galois over~$\Q$ then its Galois closure is $K_x(x)$.

\end{itemize}

We will use the following lemmas. 
Recall that~$D_x$ denotes the fundamental discriminant of a singular modulus~$x$, see Subsection~\ref{ssprelim}.

\begin{lemma}
\label{ltwoconj}
Let $x,x',y,y'$ be singular moduli. Assume that
$$
\Delta_x=\Delta_{x'}, \quad \Delta_y=\Delta_{y'}.
$$
Furthermore, assume that ${\Q(x,x')=\Q(y,y')}$. Then we have the following.
\begin{enumerate}
\item
If ${D_x\ne D_y}$ then ${\Q(x)=\Q(y)}$.

\item
If ${D_x=D_y}$ then ${K(x)=K(y)}$, where ${K=K_x=K_y}$ is the common CM-field for~$x$ and~$y$.
\end{enumerate}
\end{lemma}

\begin{proof}
The case ${D_x\ne D_y}$ is \cite[Corollary~3.3]{FR18}.

Now assume that ${D_x=D_y}$. We use the terminology of \cite[Section~3]{FR18}. If the field ${L=\Q(x,x')=\Q(y,y')}$ is $2$-elementary (that is, Galois over~$\Q$ with Galois group of  the type ${\Z/2\Z\times \cdots\times \Z/2\Z}$), then, arguing as in the beginning of the proof of \cite[Corollary~3.3]{FR18}, we obtain ${\Q(x)=\Q(y)}$.

Now assume that~$L$ is not $2$-elementary. If it is Galois over~$\Q$, then it is the Galois closure of both $\Q(x)$ and $\Q(y)$. Since the Galois closure of $\Q(x)$ is $K(x)$ and that of $\Q(y)$ is $K(y)$,  we are done. Finally, if~$L$ is not Galois over~$\Q$ then ${x'\in \Q(x)}$ and ${y'\in \Q(y)}$, and so ${L=\Q(x)=\Q(y)}$.
\end{proof}

\begin{lemma}
\label{lsamefield}
Let $x,y$ be singular moduli with the same fundamental discriminant ${D=D_x=D_y}$, and let ${K=K_x=K_y=\Q(\sqrt D)}$ be their common CM-field. Assume that ${K(x)=K(y)\ne K}$. Then we have one of the following options.

\begin{enumerate}
\item We have ${\Delta_x=\Delta_y}$, hence the singular moduli~$x$ and~$y$ are conjugate over~$\Q$. 

\item
\label{idisdel}
Up to swapping~$x$ and~$y$, we have  ${\Delta_x=4\Delta_y}$ and ${\Delta_y\equiv 1 \bmod 8}$.
\end{enumerate}
\end{lemma}

\begin{proof}
See \cite[Proposition~4.3]{ABP15}, where everything is proved except that in option~\ref{idisdel} we have ${\Delta_y\equiv 1 \bmod 8}$. For the latter, see \cite[page~407]{BLP16}.

To be precise, both in~\cite{ABP15} and~\cite{BLP16} the slightly stronger assumption ${\Q(x)=\Q(y)}$ (in our notation) is made, but the argument works also
under the hypothesis ${K(x)=K(y)}$. 
\end{proof}

\subsection{Dominant and subdominant singular moduli}
\label{ssdom}

It is well-known (see, for instance, \cite[Proposition~2.5]{BLP16} and the references therein) that
there is a one-to-one correspondence between the singular moduli of discriminant~$\Delta$ and the set~$T_\Delta$ of triples $(a,b,c)$ of integers satisfying ${b^2-4ac=\Delta}$ and
\begin{equation*}
\begin{gathered}
\gcd(a,b,c)=1, \quad \Delta=b^2-4ac,\\
\text{either\quad $-a < b \le a < c$\quad or\quad $0 \le b \le a = c$}.
\end{gathered}
\end{equation*}
If  ${(a,b,c)\in T_\Delta}$, then ${(b+\sqrt{\Delta})/2a}$ belongs to the standard fundamental domain, and the corresponding singular modulus is
${j((b+\sqrt{\Delta})/2a)}$.

We call a singular modulus  \textsl{dominant} if in the corresponding triple $(a,b,c)$ we have ${a=1}$, and \textsl{subdominant} if ${a=2}$. The following  property will be crucial.

\begin{proposition}[{\cite[Proposition 2.6]{BLP16}}]
\label{pdom}
There exist exactly one dominant and at most two subdominant singular moduli of  given discriminant~$\Delta$. More precisely,
\begin{itemize}
\item
there exist exactly~$2$  subdominant singular moduli of discriminant~$\Delta$ if ${\Delta\equiv 1\bmod 8}$, ${\Delta\ne -7}$;

\item
there exists exactly~$1$  subdominant singular modulus of discriminant~$\Delta$ if ${\Delta\equiv 8,12\bmod 16}$, ${\Delta\ne -4,-8}$;

\item
there are no subdominant singular moduli of discriminant~$\Delta$ if ${\Delta\equiv 5\bmod 8}$ or ${\Delta\equiv 0,4\bmod 16}$.

\end{itemize}
\end{proposition}

The inequality
$$
\bigl||j(z)|-e^{2\pi \Im z}\bigr|\le 2079,
$$
holds true for every~$z$ in the standard fundamental domain. It is proven in \cite[Lemma~1]{BMZ13}, but it can also be easily deduced from~\eqref{ejm} by setting ${v=\sqrt3/2}$ therein.  In particular, if~$x$ is a singular modulus 
corresponding to the triple ${(a,b,c)\in T_{\Delta_x}}$  then
$$
\bigl||x|-e^{\pi|\Delta_x|^{1/2}/a}\bigr|\le 2079.
$$
This implies that
\begin{align}
\label{euniv}
|x|&\le e^{\pi|\Delta_x|^{1/2}}+ 2079&& \text{in any case};\\
\label{eifdom}
|x|&\ge e^{\pi|\Delta_x|^{1/2}}- 2079&& \text{if $x$ is dominant};\\
\label{eifnotdom}
|x|&\le e^{\pi|\Delta_x|^{1/2}/2}+ 2079&& \text{if $x$ is not dominant};\\
\label{eifnotdomsubdom}
|x|&\le e^{\pi|\Delta_x|^{1/2}/3}+ 2079&& \text{if $x$ is neither dominant nor subdominant}.
\end{align}
These inequalities will be systematically used in the sequel, sometimes without special reference.

Note that if~$x$ is dominant, then it exceeds in absolute value any non-dominant singular modulus of the same discriminant: this is clear if ${h(\Delta_x)=1}$, and when ${h(\Delta_x)\ge 2}$, we have ${|\Delta_x|\ge 15}$, and the right-hand side of ~\eqref{eifdom} is bigger than that of~\eqref{eifnotdom}. 
This implies, in particular, that a dominant singular modulus must be real, because it cannot be equal in absolute value to any of its $\Q$-conjugates. 



\section{The primitive element}
\label{sprim}
In this section we prove Theorem~\ref{thalpha}. Let $x,y$  be singular moduli and~$\alpha$ a rational number, ${\alpha \ne 0,\pm1}$, such that ${\Q(x+\alpha y)}$ is a proper subfield of $\Q(x,y)$.

Let~$L$ be the Galois closure of $\Q(x,y)$ over~$\Q$, and denote ${G=\Gal(L/\Q)}$. Since ${\Q(x+\alpha y)\ne \Q(x,y)}$,  there exists ${\sigma \in G}$ such that
\begin{equation}
\label{esuit}
x\ne x^\sigma, \qquad y\ne y^\sigma, \qquad x+\alpha y=x^\sigma+\alpha y^\sigma.
\end{equation}
Rewriting the latter equality as
\begin{equation}
\label{emain}
x-x^\sigma=-\alpha (y-y^\sigma),
\end{equation}
we obtain ${\Q(x-x^\sigma)=\Q(y-y^\sigma)}$.  It follows from Theorem~\ref{thsumdiff} that 
$$
\Q(x,x^\sigma)=\Q(y,y^\sigma).
$$
Now, using Lemmas~\ref{ltwoconj} and~\ref{lsamefield}, and swapping~$x$ and~$y$ if necessary,  we are in one of the following three options.

\begin{itemize}
\item
\textbf{Equal discriminants:} ${\Delta_x=\Delta_y}$.

{\sloppy

\item
\textbf{Equal fundamental discriminants, but distinct discriminants:}  ${D_x=D_y=:D}$ and ${K(x)=K(y)}$, where ${K=\Q(\sqrt D)}$ is the common CM field of~$x$ and~$y$;   furthermore, ${\Delta_y\equiv 1\bmod 8}$ and  ${\Delta_x=4\Delta_y}$. 

}

\item
\textbf{Distinct fundamental discriminants:}
${D_x\ne D_y}$ but  ${\Q(x)=\Q(y)}$.
\end{itemize}

We study these three cases separately.

Note that in each of the three cases above we have ${h(\Delta_x)=h(\Delta_y)}$.  We denote this quantity by~$h$. Note that 
$$
[\Q(x):\Q]=[\Q(y):\Q]=h. 
$$
In the case ${h=1}$ there is nothing to prove, and the case ${h=2}$ is very easy. Indeed, existence of~$\sigma$ with the property~\eqref{esuit} implies that ${\Q(x)=\Q(y)}$ and that~$\alpha$ is defined by~\eqref{ealpha}, so we are in the situation of Example~\ref{exquad}.

Thus, in the sequel we we assume that
${h\ge 3}$.
This will also be used systematically, usually without special reference.


\subsection{The case of equal discriminants}
\label{sseqdis}
We assume now that ${\Delta_x=\Delta_y=:\Delta}$. We may also assume that~$x$ is dominant as defined in Subsection~\ref{ssdom}.

Fix a Galois morphism~$\sigma$ satisfying~\eqref{esuit}.  Note that either ${y^\sigma\ne x}$ or ${y^{\sigma^{-1}}\ne x}$; indeed, if ${y^\sigma=y^{\sigma^{-1}}=x}$ then~\eqref{emain} implies ${\alpha=1}$, a contradiction. Thus, replacing, if necessary,~$\sigma$ by~$\sigma^{-1}$, we may assume that ${y^\sigma\ne x}$.
Using~\eqref{emain}, we obtain
\begin{equation}
\label{equote}
\alpha = - \frac{x-x^\sigma}{y-y^\sigma}, \qquad y,y^\sigma\ne x.
\end{equation}
This identity will be our principal tool.

\subsubsection{A lower bound for~$h$}

Let us first prove  that ${h\ge 4}$. As  ${h\ge 3}$, we will assume that  ${h=3}$ and derive a contradiction.    

When ${h=3}$, the field $\Q(x,y)$ is the full Ring Class Field associated to the discriminant~$\Delta$; we denote this field~$L$. In particular, it contains the imaginary quadratic CM field ${K=\Q(\sqrt\Delta)}$. Since~$x$ is dominant, it must be real, see end of Subsection~\ref{ssdom}. Hence~$y$ cannot be real, and the~$3$ singular moduli of discriminant~$\Delta$ are ${x,y,\bar y}$.

The maximal proper subfields of the field~$L$ are $\Q(x)$, $\Q(y)$, $\Q(\bar y)$ and~$K$. The element ${x+\alpha y}$ cannot  belong to~$\Q(x)$ or~$\Q(y)$ because ${y\notin \Q(x)}$ and ${x\notin \Q(y)}$. Thus, either ${x+\alpha y\in\Q(\bar y)}$ or ${x+\alpha y\in K}$.

The non-identical elements of the Galois group $\Gal(L/K)$ are the $3$-cyclic permutations of the set ${\{x,y,\bar y\}}$. In particular, there is ${\theta \in \Gal(L/K)}$ such that
$$
x^\theta=y, \quad y^\theta=\bar y, \quad \bar y^\theta= x.
$$
If ${x+\alpha y\in\Q(\bar y)}$ then ${y+\alpha \bar y=(x+\alpha y)^\theta \in \Q(x)\subset \R}$. Hence ${\alpha= -1}$, a contradiction.
And if ${x+\alpha y\in K}$ then ${(x+\alpha y)^{\theta^{-1}}= x+\alpha y}$. But we also have ${(x+\alpha y)^{\theta^{-1}}=\bar y+\alpha x}$. Hence
${\bar y-\alpha y\in \R}$, which implies ${\alpha=1}$, a contradiction.

\subsubsection{An upper bound for~$h$}
\label{ssslesix}
We already know that ${h\ge 4}$. Our next aim is proving that ${h\le 6}$. We are going to prove even more than this:~$\Delta$ satisfies one of the following conditions:
\begin{align}
\label{etwosubs}
&h(\Delta) \in \{4,5,6\}, && \Delta\equiv1\bmod 8 \\
\label{eonesub}
&h(\Delta)=4, && \Delta \equiv 8,12\bmod 16.
\end{align}
Thus, let us assume by contradiction that either ${h\ge 7}$ or ${h\in \{4,5,6\}}$ and none of  conditions~\eqref{etwosubs},~\eqref{eonesub} is satisfied.
Note that, since ${h\ge 4}$, we have
\begin{equation}
\label{edelge39}
|\Delta|\ge 39.
\end{equation}


Let~$\sigma$ be as in~\eqref{equote}. Since~$x$ is dominant, but neither~$x^\sigma$ nor~$y$ nor~$y^\sigma$ is, we use~\eqref{eifdom},~\eqref{eifnotdom},~\eqref{equote} and~\eqref{edelge39} to obtain the lower estimate
\begin{equation}
\label{elower}
|\alpha|\ge
\frac{e^{\pi|\Delta|^{1/2}}-e^{\pi|\Delta|^{1/2}/2}-4178}{2e^{\pi|\Delta|^{1/2}/2}+4178} \ge 0.448
e^{\pi|\Delta|^{1/2}/2}.
\end{equation}

The group ${H=\Gal(L/\Q(x))}$ is a subgroup of the group ${G=\Gal(L/\Q)}$ of index ${h=[\Q(x):\Q]}$. Call ${\gamma \in G}$ \textsl{suitable} if neither~$x^\gamma$ nor~$x^{\sigma\gamma}$ is dominant or subdominant. We claim that a suitable~$\gamma$ exists unless~$\Delta$ satisfies one of  conditions~\eqref{etwosubs},~\eqref{eonesub}.

Since  there exist exactly one dominant and at most~$2$ subdominant singular moduli of discriminant~$\Delta$ (see Proposition~\ref{pdom}),   there may exist at most~$3$ cosets in  $H\backslash G$  sending~$x$  to a dominant or a subdominant element. Similarly, there exist at most~$3$ cosets in ${\sigma^{-1}H\sigma \backslash G}$ sending~$x^\sigma$  to a dominant or a subdominant conjugate.  The total cardinality of these cosets does not exceed $6|H|$. Hence a suitable~$\gamma$ exists if ${h\ge 7}$.

Using Proposition~\ref{pdom}, the same holds true if none of  conditions~\eqref{etwosubs},~\eqref{eonesub} is satisfied. Indeed, if ${\Delta\not\equiv 1\bmod8}$ then there is at most one subdominant conjugate. This means that  we have at most~$4$ ``bad'' cosets, and we find a suitable~$\gamma$ if ${h\ge 5}$.

Finally, if ${\Delta\not\equiv 1\bmod8}$ and ${\not\equiv 8,12\bmod16}$ then~$\Delta$ does not admit subdominant singular moduli at all. Hence in this case we have only~$2$ ``bad'' cosets, and we find a suitable~$\gamma$ if ${h\ge 3}$.

Thus, a suitable~$\gamma$ exists. From~\eqref{equote} we deduce that
\begin{equation}
\label{ewithgamma}
\alpha = - \frac{x^\gamma-x^{\sigma\gamma}}{y^\gamma-y^{\sigma\gamma}}.
\end{equation}
Since neither~$x^\gamma$ nor~$x^{\sigma\gamma}$ is dominant or subdominant, we may use~\eqref{eifnotdomsubdom} and~\eqref{edelge39} to estimate
\begin{equation}
\label{eupperbefore}
|\alpha|\le\frac{2e^{\pi|\Delta|^{1/2}/3}+4178}{|y^\gamma-y^{\sigma\gamma}|} \le\frac{8.04e^{\pi|\Delta|^{1/2}/3}}{|y^\gamma-y^{\sigma\gamma}|}.
\end{equation}
Theorem~\ref{thseparweak} implies that ${|y^\gamma-y^{\sigma\gamma}|\ge 800|\Delta|^{-4}}$.
Hence
$$
|\alpha|\le 0.0101|\Delta|^4e^{\pi|\Delta|^{1/2}/3}.
$$
Comparing this and~\eqref{elower}, we obtain
$
{e^{\pi|\Delta|^{1/2}/6}\le 0.0226|\Delta|^4}
$.
This inequality is contradictory for ${|\Delta|\ge 3000}$.

Thus, ${|\Delta|< 3000}$. We again use~\eqref{eupperbefore}, but this time we apply Proposition~\ref{psmaldis} to estimate ${|y^\gamma-y^{\sigma\gamma}|\ge 3.07}$. We obtain
$
{|\alpha| \le 2.62 e^{\pi|\Delta|^{1/2}/3}}
$.
Comparing this with~\eqref{elower}, we obtain ${|\Delta|< 12}$, a contradiction.

\subsubsection{The remaining~$\Delta$}
\label{sss38}
We are left with~$\Delta$ satisfying one of  conditions~\eqref{etwosubs},~\eqref{eonesub}. There are 38 such discriminants, their full list (found using the  \textsf{SAGE} function \textsf{cm\_orders}) being
\begin{equation}
\label{ebaddeltas}
\begin{aligned}
&\mathbf{-}39, -47, -55, -56, -63, -68, -79, \mathbf{-84}, -87, -103, \mathbf{-120}, -127, \\
&\mathbf{-132}, -135, -136, \mathbf{-168}, -175, \mathbf{-180}, -184, -196, -207, \mathbf{-228}, \\
&\mathbf{-}247, \mathbf{-280}, -292, \mathbf{-312}, -328, \mathbf{-340, -372}, -388, \mathbf{-408, -520}, \\
&\mathbf{-532}, -568, \mathbf{-708, -760}, -772, \mathbf{-1012}.
\end{aligned}
\end{equation}

Note that~16 discriminants are bold-faced. Those are of class number~$4$ and  class group of type $[2,2]$.  If~$\Delta_x$ has this property then $\Q(x)/\Q$ is a Galois extension  (see, for instance, \cite[Corollary~3.3]{ABP15}).

Let~$\Delta$ be from the list~\eqref{ebaddeltas}, and let ${x_1,\ldots, x_h}$ be the singular moduli of discriminant~$\Delta$, with ${x=x_1}$ dominant.
It follows from~\eqref{equote} that either~$\alpha$ or $-\alpha$ belongs to the set\footnote{Here and below~$j$ is used as a running index, not as the $j$-invariant.}
\begin{equation}
\label{esetofalphas}
A_\Delta=\left\{\frac{x_1-x_i}{x_j-x_k}: 2\le i,j\le h, \ j<k\le h\right\}.
\end{equation}
Using \textsf{PARI}, we can show that this set does not contain rational numbers. For those~22 discriminants which are not bold-faced, we even show that~$A_\Delta$ does not contain real numbers. To be precise, using a simple PARI script, we are able to show that
$$
\min \{|\Im z|: z\in A_\Delta\}\ge 345,
$$
for every~$\Delta$ in the list~\eqref{ebaddeltas} except for the bold-faced ones.

For the  bold-faced~$\Delta$, this argument does not work, because all their singular moduli are real. However, since in these cases $\Q(x)$ is Galois
over~$\Q$, all the singular moduli are contained in $\Q(x)$. 
 Hence we may write, in a unique way, ${x_i=f_i(x)}$, each $f_i$ being a polynomial of degree not exceeding~$3$ (recall that ${h=4}$ for all the bold-faced discriminants). It is easy to verify, using \textsf{PARI}, that the polynomials ${f_1-f_i}$ and ${f_j-f_k}$ are not proportional for every choice of $i,j,k$ as above, showing that there are no rational numbers in  $A_\Delta$.

This rules out all~$\Delta$ from~\eqref{ebaddeltas}, completing the proof of Theorem~\ref{thalpha} in the case of equal discriminants.

\subsection{Equal fundamental discriminants, but distinct discriminants}

Now assume that ${D_x=D_y}$, but ${\Delta_x\ne \Delta_y}$. In this case, as we have seen at the beginning of Section~\ref{sprim}, we have ${\{\Delta_x,\Delta_y\}=\{\Delta, 4\Delta\}}$, where ${\Delta\equiv 1\bmod 8}$. We may assume that
$$
\Delta_x=4\Delta, \qquad \Delta_y=\Delta,
$$
and that~$x$ is dominant. Since ${h(\Delta)\ge 3}$, we have
\begin{equation}
\label{edelge23}
|\Delta|\ge 23.
\end{equation}
Under the assumption  
${\Q(x,y)\ne \Q(x+\alpha y)}$,  we can find, as before, an element ${\sigma \in G=\Gal(L/\Q)}$ such that
\begin{equation}
\label{equotebis}
\alpha = - \frac{x-x^\sigma}{y-y^\sigma}. 
\end{equation}
Since~$x$ is dominant and~$x^\sigma$ is not, we use~\eqref{euniv},~\eqref{eifdom},~\eqref{eifnotdom} and~\eqref{edelge23} to obtain the estimate
\begin{align}
|\alpha|&\ge \frac{e^{\pi|\Delta_x|^{1/2}}-e^{\pi|\Delta_x|^{1/2}/2}-4178}{e^{\pi|\Delta_y|^{1/2}}+e^{\pi|\Delta_y|^{1/2}/2}+4178} \nonumber \\
&=\frac{e^{2\pi|\Delta|^{1/2}}-e^{\pi|\Delta|^{1/2}}-4178}{e^{\pi|\Delta|^{1/2}}+e^{\pi|\Delta|^{1/2}/2}+4178}
\nonumber\\
\label{elowerdist}
&\ge 0.998e^{\pi|\Delta|^{1/2}}.
\end{align}

Next, as in Subsection~\ref{ssslesix}, we want to find ${\gamma\in G}$ such that neither~$x^\gamma$ nor~$x^{\sigma\gamma}$ is dominant or subdominant. This time, however, the task is much easier: since ${\Delta\equiv 1\bmod 8}$, we have ${\Delta_x=4\Delta\equiv 4\bmod 32}$, and Proposition~\ref{pdom} implies that there are no subdominant singular moduli of discriminant~$\Delta_x$. Hence we only have to assure that  neither~$x^\gamma$ nor~$x^{\sigma\gamma}$ is dominant, and such~$\gamma$ exists as soon as ${[G:H]=h\ge 3}$, which is our assumption.

We again have
\begin{equation*}
\alpha = - \frac{x^\gamma-x^{\sigma\gamma}}{y^\gamma-y^{\sigma\gamma}}.
\end{equation*}
By~\eqref{eifnotdomsubdom} and~\eqref{edelge23}, we obtain
\begin{equation}
\label{euuu}
|\alpha|\le \frac{2e^{\pi|\Delta_x|^{1/2}/3}+4178}{|y^\gamma-y^{\sigma\gamma|}}= \frac{2e^{2\pi|\Delta|^{1/2}/3}+4178}{|y^\gamma-y^{\sigma\gamma|}} \le \frac{2.19e^{2\pi|\Delta|^{1/2}/3}}{|y^\gamma-y^{\sigma\gamma|}}
\end{equation}
Theorem~\ref{thseparweak} implies that ${|y^\gamma-y^{\sigma\gamma}|\ge 800|\Delta_y|^{-4}=800|\Delta|^{-4}}$.
Hence
\begin{equation*}
|\alpha|\le 0.00274|\Delta|^4e^{2\pi|\Delta|^{1/2}/3}.
\end{equation*}
Comparing this with~\eqref{elowerdist}, we obtain
${e^{\pi|\Delta|^{1/2}/3}\le 0.00275|\Delta|^4}$.
This inequality implies that  ${|\Delta|< 300}$, 
 in which case ${|y^\gamma-y^{\sigma\gamma}|\ge 92.4}$ by Proposition~\ref{psmaldis}. Together with~\eqref{euuu} this implies ${|\alpha|\le 0.0238e^{2\pi|\Delta|^{1/2}/3}}$, which contradicts~\eqref{elowerdist}.
This completes the proof in the case of equal fundamental discriminants, but distinct discriminants.

\subsection{Distinct fundamental discriminants}

Now we assume that ${D_x\ne D_y}$. Since in this case we have ${\Q(x)=\Q(y)}$, we may use Corollary~4.2 of~\cite{ABP15}, where all couples of singular moduli $(x,y)$ such that ${\Q(x)=\Q(y)}$ but ${D_x\ne D_y}$ are classified. Since ${h\ge 3}$, our $\Delta_x$ and $\Delta_y$ are featured in the six bottom lines of Table~2 on page~12 of~\cite{ABP15}. To be precise,  there are 15 (up to swapping~$x$ and~$y$) possible pairs $(\Delta_x,\Delta_y)$:
\begin{equation*}
\begin{array}{llll}
(-96,-192), & (-96, -288), & (-120, -160) , & (-120, - 280), \\
(-120, -760), & (-160, -280), & (-160, - 760),  & (-180, -240), \\
(-192,-288), & (-195, -520), & (-195, - 715), & (-280, -760), \\
(-340, - 595), & (-480, - 960),  & (-520, - 715).
\end{array}
\end{equation*}
All of them can be checked using a \textsf{PARI} script in the same fashion as  the bold-faced discriminants in Subsection~\ref{sss38}.

To be precise, in all of these cases the field ${\Q(x)=\Q(y)}$ is Galois over~$\Q$. Hence the conjugates ${x_1, \ldots, x_h}$ of~$x$ and the conjugates ${y_1, \ldots, y_h}$ of~$y$ can be uniquely expressed as
${x_i=f_i(x)}$ and ${y_i=g_i(x)}$,
where~$f_i$ and~$g_i$ are polynomials over~$\Q$ of degree not exceeding ${h-1}$. Now, using \textsf{PARI}, it is easy to verify that in each case any of the polynomials ${f_1-f_i}$ is not proportional to any of the polynomials ${g_j-g_k}$. This rules out all the 15 pairs in the list above, completing the proof of Theorem~\ref{thalpha}.

\paragraph{Acknowledgments}
The authors thank Sasha Borichev, Lars Kühne and Ricardo Me\-na\-res for very useful discussions. They also thank Yulin Cai for careful reading of the manuscript and detecting a number of inaccuracies. 

We are indebted to the referee for very careful reading of the manuscript, correcting a mistake in our proof of Proposition~\ref{pbigimwcorrected}, and making many suggestions that helped us to improve the presentation.

All calculations were performed using \textsf{PARI}~\cite{pari} or \textsf{SAGE}~\cite{sagemath}.  We thank Bill Allombert and Karim Belabas for the \textsf{PARI} tutorial. Our \textsf{PARI} scripts can be viewed here: \newline
\url{https://github.com/yuribilu/Separating/blob/master/scripts.gp}.

\paragraph{Funding}
Yu. B.'s work on this article profited from attending the Valparaiso 2019 conference ``Explicit Number Theory'', which was sponsored by the Ecos Sud/Conicyt project C17E01. He was also partially supported by  the SPARC Project P445 (India). 

B.~F. was partially supported by the IRN GANDA. 

H.~Z.  was partially supported by China National Science Foundation Grant (No. 11501477), the Fundamental Research Funds for the Central Universities (No. 20720170001) and the Science Fund of Fujian Province (No. 2015J01024).

{\footnotesize

\bibliographystyle{amsplain}
\bibliography{primel}

\bigskip

\noindent
\textbf{Yuri BILU}: Institut de Mathématiques de Bordeaux, Université de Bordeaux et CNRS, 351 cours de la Libération, 33405 Talence CEDEX, France, and
School of Mathematical Sciences,
Xiamen University,
Xiamen City,
Fujian Province, P.R.China; \url{yuri@math.u-bordeaux.fr}

\noindent
\textbf{Bernadette FAYE}: Université Gaston-Berger de Saint-Louis, UFR SAT,
BP: 234,
Saint Louis, Senegal; \url{bernadette.fayee@gmail.com}

\noindent
\textbf{Huilin ZHU}:  School of Mathematical Sciences,
Xiamen University,
Xiamen City,
Fujian Province, P.R.China; \url{hlzhu@xmu.edu.cn}
(corresponding author)

}

\end{document}